\pdfoutput=1

\documentclass[12pt]{amsart}

\usepackage{color}
\usepackage{graphicx}
\usepackage[mathcal, mathscr]{euscript}
\usepackage{bbm}
\usepackage{float}
\usepackage{amsmath,amsthm,amssymb,amsfonts,amscd,epsfig,latexsym,graphicx,textcomp}
\usepackage{tikz-cd}
\usepackage{psfrag}
\usepackage[all]{xy}
\xyoption{pdf}
\usepackage{stackengine}
\usepackage{comment}
\usepackage{hhline}
\usepackage{diagbox}
\usepackage{float}
\usepackage{caption}
\usepackage{eufrak}
\usepackage{slashed}
\usepackage{amsmath, amssymb, graphics, setspace}

\usepackage{listings}

\usepackage{tikz}
\usepackage{stmaryrd}

\restylefloat{figure}

\usepackage{thmtools}
\usepackage{thm-restate}

\usepackage{hyperref}

\usepackage{cleveref}

\newtheorem{theorem}{Theorem}[section]
\newtheorem{lemma}[theorem]{Lemma}
\newtheorem{corollary}[theorem]{Corollary}

\newtheorem{question}[theorem]{Question}

\theoremstyle{definition}
\newtheorem{definition}[theorem]{Definition} 
\newtheorem{example}[theorem]{Example} 

\theoremstyle{remark}
\newtheorem{remark}[theorem]{Remark}
\numberwithin{equation}{section}
\newcommand{\ot}{\otimes}
\newcommand{\ra}{\rightarrow}

\newcommand{\BC}{\mathbb{C}}

\newcommand{\IIDiag}{\raisebox{-0.33\height}{\includegraphics[scale=0.25]{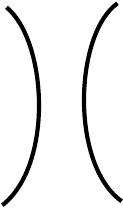}}}

\newcommand{\XDiag}{\raisebox{-0.33\height}{\includegraphics[scale=0.25]{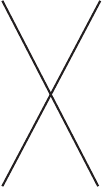}}}
\newcommand{\PMEdgeDiag}{\raisebox{-0.33\height}{\includegraphics[scale=0.25]{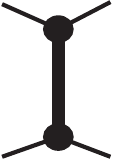}}}

\newcommand{\SquareVirtual}{\raisebox{-0.33\height}{\includegraphics[scale=0.25]{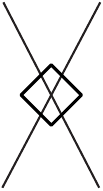}}}

\newcommand{\NodeVirtual}{\raisebox{-0.33\height}{\includegraphics[scale=0.25]{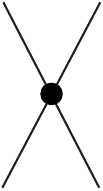}}}

\newcommand{\DoubleTheta}{\raisebox{-0.33\height}{\includegraphics[scale=0.5]{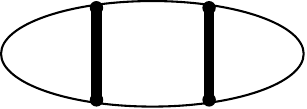}}}
\newcommand{\DoubleThetaZero}{\raisebox{-0.33\height}{\includegraphics[scale=0.5]{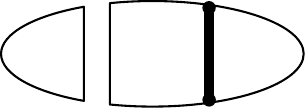}}}
\newcommand{\DoubleThetaOne}{\raisebox{-0.33\height}{\includegraphics[scale=0.5]{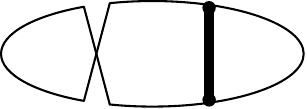}}}
\newcommand{\ThetaG}{\raisebox{-0.33\height}{\includegraphics[scale=0.5]{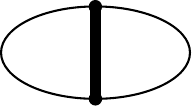}}}
\newcommand{\ThetaZero}{\raisebox{-0.33\height}{\includegraphics[scale=0.5]{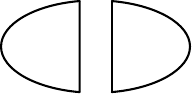}}}
\newcommand{\ThetaOne}{\raisebox{-0.33\height}{\includegraphics[scale=0.5]{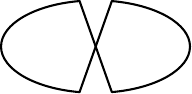}}}

\newcommand{\ThetaZeroSquiggle}{\raisebox{-0.33\height}{\includegraphics[scale=0.5]{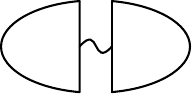}}}
\newcommand{\ThetaOneSquiggle}{\raisebox{-0.33\height}{\includegraphics[scale=0.5]{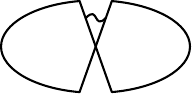}}}

\newcommand{\DoubleThetaB}{\raisebox{-0.33\height}{\includegraphics[scale=0.5]{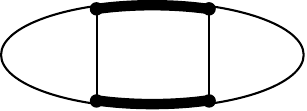}}}
\newcommand{\DoubleThetaBOne}{\raisebox{-0.33\height}{\includegraphics[scale=0.5]{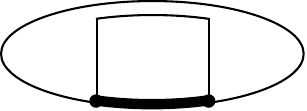}}}
\newcommand{\DoubleThetaBTwo}{\raisebox{-0.33\height}{\includegraphics[scale=0.5]{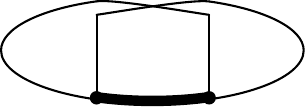}}}
\newcommand{\DoubleThetaBThree}{\raisebox{-0.33\height}{\includegraphics[scale=0.5]{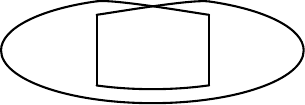}}}
\newcommand{\DoubleThetaBFour}{\raisebox{-0.33\height}{\includegraphics[scale=0.5]{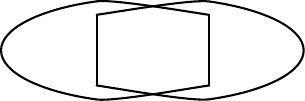}}}

\newcommand{\KThreeThree}{\raisebox{-0.4\height}{\includegraphics[scale=0.35]{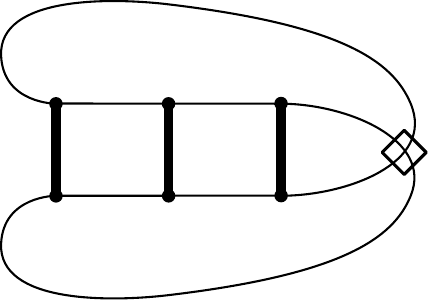}}}
\newcommand{\KThreeThreeOne}{\raisebox{-0.4\height}{\includegraphics[scale=0.35]{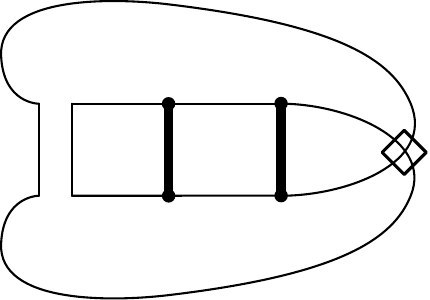}}}
\newcommand{\KThreeThreeTwo}{\raisebox{-0.4\height}{\includegraphics[scale=0.35]{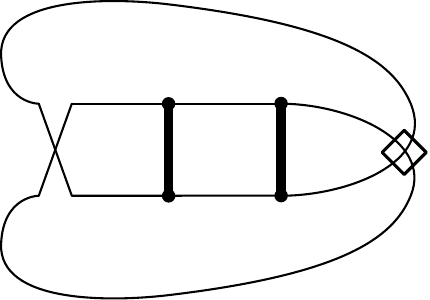}}}
\newcommand{\KThreeThreeThree}{\raisebox{-0.4\height}{\includegraphics[scale=0.35]{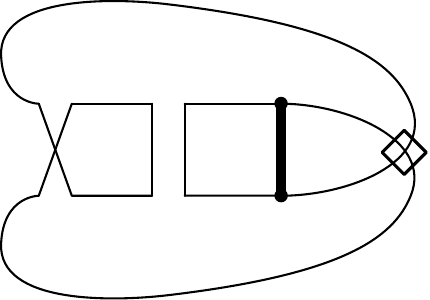}}}
\newcommand{\KThreeThreeFour}{\raisebox{-0.4\height}{\includegraphics[scale=0.35]{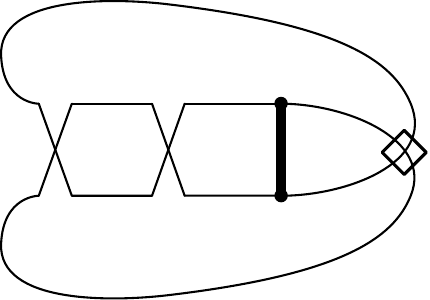}}}
\newcommand{\KThreeThreeFive}{\raisebox{-0.4\height}{\includegraphics[scale=0.35]{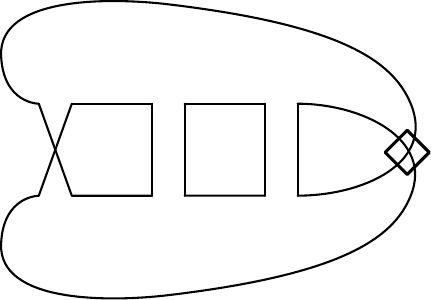}}}
\newcommand{\KThreeThreeSix}{\raisebox{-0.4\height}{\includegraphics[scale=0.35]{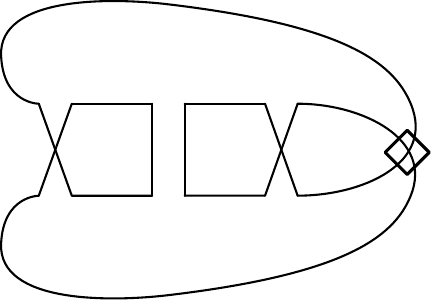}}}
\newcommand{\KThreeThreeSeven}{\raisebox{-0.4\height}{\includegraphics[scale=0.35]{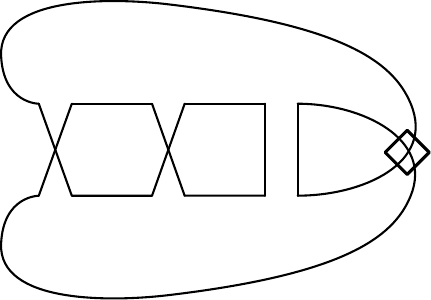}}}
\newcommand{\KThreeThreeEight}{\raisebox{-0.4\height}{\includegraphics[scale=0.35]{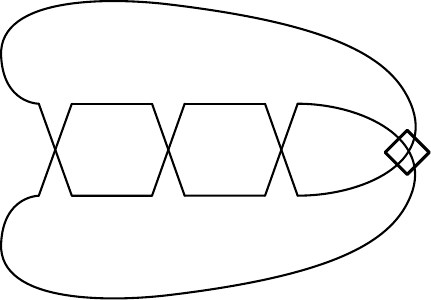}}}

\newcommand{\BN}{\mathbb{N}}

\newcommand{\NN}{{\mathbb N}}

 \newcommand{\CGlyph}{\raisebox{-0.25\height}{\includegraphics[width=0.5cm]{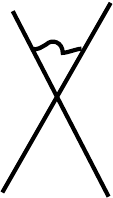}}}
\newcommand{\VGlyph}{\raisebox{-0.25\height}{\includegraphics[width=0.5cm]{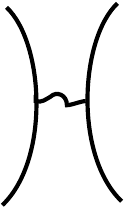}}}

\newcommand{\YMTens}{\raisebox{-0.25\height}{\includegraphics[width=0.5cm]{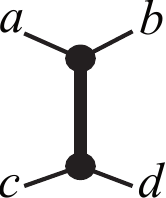}}}
\newcommand{\IMTens}{\raisebox{-0.25\height}{\includegraphics[width=0.5cm]{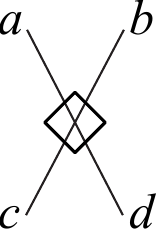}}}

\newcommand{\BigYMTens}{\raisebox{-0.4\height}{\includegraphics[width=1cm]{LabeledPMEdge.pdf}}}
\newcommand{\BigIMTens}{\raisebox{-0.4\height}{\includegraphics[width=1cm]{SquareVirtualCrossLabeled.pdf}}}

\newcommand{\CTens}{\raisebox{-0.25\height}{\includegraphics[width=0.5cm]{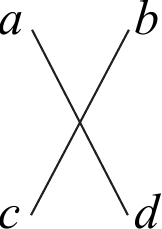}}}
\newcommand{\FTens}{\raisebox{-0.25\height}{\includegraphics[width=0.5cm]{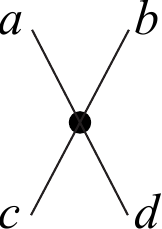}}}

\def\del{\partial}

\def\del{\partial}

\usepackage{xpatch}
\makeatletter   
\xpatchcmd{\@tocline}
{\hfil\hbox to\@pnumwidth{\@tocpagenum{#7}}\par}
{\ifnum#1<0\hfill\else\dotfill\fi\hbox to\@pnumwidth{\@tocpagenum{#7}}\par}
{}{}
\makeatother    

\makeatletter
 \def\l@subsection{\@tocline{2}{0pt}{4pc}{6pc}{}}
\def\l@subsubsection{\@tocline{3}{0pt}{8pc}{8pc}{}}
 \makeatother

\begin{document}

\title{A State Sum for the Total Face Color Polynomial}

\thanks{}

\author{Scott Baldridge}
\address{Department of Mathematics, Louisiana State University,
Baton Rouge, LA}
\email{baldridge@math.lsu.edu}

\author{Louis H. Kauffman}
\address{Department of Mathematics, Statistics and Computer Science, 851 South Morgan Street, University of Illinois at Chicago,
Chicago, Illinois 60607-7045} 

\author{Ben McCarty}
\address{Department of Mathematical Sciences, University of Memphis,
Memphis, TN}
\email{ben.mccarty@memphis.edu}

\date{}

\begin{abstract} 
The total face color polynomial is based upon the Poincar\'{e} polynomials of a family of filtered $n$-color homologies.  It counts the number of $n$-face colorings of ribbon graphs for each positive integer $n$.  As such, it may be seen as a successor of the Penrose polynomial, which at $n=3$ counts $3$-edge colorings (and consequently $4$-face colorings) of planar trivalent graphs.  In this paper we describe a state sum formula for the polynomial.  This formula unites two different perspectives about graph coloring:  one based upon topological quantum field theory and the other on diagrammatic tensors.
\end{abstract}

\maketitle

\section{Introduction}

The $2$-variable total face color polynomial of a ribbon graph $\Gamma$, $T(\Gamma,n,t)$, was introduced in 2023 \cite{BM-Color} as the Poincar\'{e} polynomial in $t$ of the filtered $n$-color homology, which exists at the top level of a robust family of homology theories for trivalent ribbon graphs.  To tell that story, we start at the base of that superstructure.  

We begin with a {\em perfect matching graph} $\Gamma_M$, which is a ribbon graph $\Gamma$ of a trivalent graph $G(V,E)$ together with a perfect matching $M\subset E$.  Here a {\em ribbon graph} $\Gamma$ is thought of as the closure of a small neighborhood of the $1$-skeleton $G$ of a CW complex of a closed surface $\overline{\Gamma}$ together with the $1$-skeleton  (cf. \Cref{sec:ribbonGraphs} for details).  In \cite{BM-Color}, the first and third authors create a state system based on the recursive relations of the Penrose polynomial:
\begin{eqnarray*}
\left[ \PMEdgeDiag \right]_n & = & \left[ \IIDiag  \right]_n  - \left[ \XDiag \right]_n\\
\left[ \bigcirc \right]_n & = & n.
\end{eqnarray*}
This state system is then used to develop a spectral sequence whose $E_1$ page is a nontrivial bigraded homology theory, analogous Khovanov homology \cite{Kho}, and whose Euler characteristic is the evaluation of the Penrose polynomial at $n$ \cite{BM-Color}.  

The $E_\infty$ page of that spectral sequence is a filtered homology theory, analogous to Lee homology \cite{LeeHomo}, which after an appropriate change of basis (cf. \Cref{def:colorbasis}) is seen to be generated by proper face colorings of certain ribbon graphs with $n$ colors.  By taking the Poincar\'{e} polynomial of the filtered theory, instead of the Euler characteristic, one obtains a stronger invariant, $T(\Gamma_M,n,t)$ (cf. \Cref{definition:totalfacecolorpolynomial}).  The total face color polynomial derives its name from the fact that for any trivalent graph $G$ with $Aut(G) = 1$, after evaluating at $t=1$, the polynomial counts the number of distinct face colorings of all possible ribbon graphs for $G$ with $n$ colors (cf. Section 7-8 of \cite{BM-Color}).  

Because of this fact, define $T(\Gamma_M,n) :=T(\Gamma_M,n,1)$, and call this the total face color polynomial of $\Gamma_M$.  This polynomial is the sum of the Betti numbers for each $n$ and is therefore one of the simplest invariants one can obtain from this family of homologies.  It is shown to be a natural successor to the Penrose polynomial in that it is equal to the Penrose polynomial for planar graphs (see Theorem F of \cite{BM-Color} for example).   Unfortunately, unlike the Euler characteristic of a homology, the Poincar\'{e} polynomial on which this polynomial is based requires computation of the entire homology and cannot be computed at the chain level.  The first and third author computed the polynomial for numerous examples, each of which required computing the homology for multiple values of $n$ to get enough data points to compute the polynomial.  For example, Theorem 7.9 of \cite{BM-Color} , it was  shown that computation of the polynomial required computation of as many as $e+f-1$ different filtered $n$-color homologies of a graph $G$ where $e$ is the number of edges and $f$ is the number of faces.

The total face color polynomial and its relation to the Penrose polynomial depends upon the underlying topological quantum field theory (TQFT), spectral sequences, bigraded homology, and more.  However, despite the richness of the theory on which it sits, the total face color polynomial leads to an abstract graph invariant for trivalent graphs.  In particular, if $G$ is trivalent, then we define $T(G,n)$ to be the total face color polynomial of the blowup of a ribbon graph of $G$ (see the paragraph after \Cref{definition:totalfacecolorpolynomial} for why it is an abstract graph invariant). Like other abstract graph invariants (e.g. Tutte polynomial, etc.), there was some hope that a state sum or deletion-contraction formula could be found to compute it.  Remarkably, we will show that this is provided by the Penrose-Kauffman bracket, which we discuss next.  

The Penrose polynomial naturally extends to nonplanar ribbon graphs, but despite a great deal of study in the literature (for example \cite{Aigner, EMM, Jaeger,Martin}), one problem persisted:  the polynomial does not necessarily count the number of $3$-edge colorings at $n=3$ when the graph is not planar (cf. \Cref{ex:K33}).  In 2015, the second author \cite{Kauffman} created a state sum, which is equivalent to the Penrose polynomial (cf. \Cref{defn:penrose_poly}) evaluated at $n=3$, when the graph is planar.  His bracket took the Penrose relations above and incorporated one additional relation involving only the virtual crossings of the original ribbon graph, which are marked with a square,   
\begin{eqnarray*}
\left\llbracket \SquareVirtual \right\rrbracket & = & 2 \left\llbracket \NodeVirtual  \right\rrbracket  - \left\llbracket \XDiag \right\rrbracket.
\end{eqnarray*}
The second author showed that the modified bracket correctly computes the number of $3$-edge colorings of nonplanar graphs, thus generalizing Penrose's result for planar graphs. The Penrose-Kauffman bracket extends to a polynomial for each positive integer, $n\in \NN$ (which is recorded for the first time in this paper in \Cref{defn:PenroseKauffman}), but its proper interpretation for values of $n>3$ remained mysterious, until it was linked to the total face color polynomial.  This is the main theorem of the paper:

\begin{theorem}
Let $G(V,E)$ be a connected trivalent graph with perfect matching $M$ and let $\Gamma_M$ be a perfect matching graph for the pair $(G,M)$.  Then
$$T(\Gamma_M,n) = \left\llbracket \Gamma_M \right\rrbracket.$$
\label{thm:PKEqualsT}
\end{theorem}

\Cref{thm:PKEqualsT} unites two perspectives on the problem of coloring perfect matching graphs: one based upon TQFTs and the other on diagrammatic tensors.  The two perspectives are highlighted in \Cref{thm:TEqualsMatching} and \Cref{thm:PKEqualsMatching}, respectively, which together prove \Cref{thm:PKEqualsT}.  Using the TQFT machinery of harmonic colorings, the color hypercube, and more from \cite{BM-Color}, we show in \Cref{thm:TEqualsMatching} that the total face color polynomial, which counts face colorings that leave the faces that correspond to the cycles of $G\setminus M$ uncolored, is equal to the count of perfect matching $n$-colorings (see \Cref{fig:PMnColor}).  \Cref{thm:PKEqualsMatching} presents a graph theoretic argument using diagrammatic tensors that the Penrose-Kauffman bracket also counts perfect matching $n$-colorings (see \Cref{def:PMnColor}).

\begin{figure}[H]
\includegraphics[scale=.75]{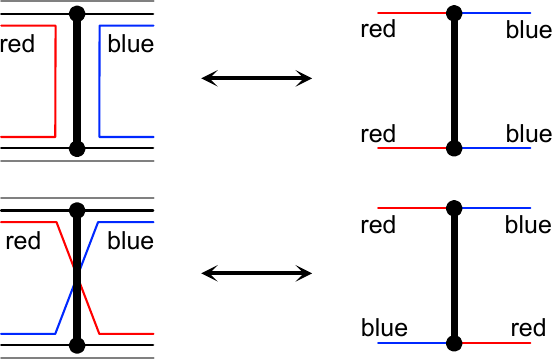}
\caption{Equivalence of colorings of faces adjacent to a perfect matching edge and perfect matching $n$-colorings.}
\label{fig:PMnColor}
\end{figure}


We encourage graph theorists to read \Cref{thm:PKEqualsMatching} first, which can be understood without needing to know TQFTs.  To fully appreciate \Cref{thm:TEqualsMatching} we encourage the reader to see \cite{BM-Color} where the machinery is fully worked out (see Theorem D in \cite{BM-Color}).  



As a corollary of \Cref{thm:PKEqualsT}, along with Theorem 6.17 and Remark 7.5 of \cite{BM-Color}, we obtain the following consequences of uniting these two perspectives.
\begin{corollary}
Let $G(V,E)$ be a connected trivalent graph.   
\begin{enumerate}
\item The total face color polynomial, $T(G,n)$, can be computed using the PK-bracket on the blow-up of any ribbon graph of $G$. 
\item The total face color polynomial gives meaning to the PK-bracket for $n>3$:  when $Aut(G)=1$, the PK-bracket is the total of the counts of all $n$-face colorings of all ribbon graphs of $G$.
\end{enumerate}
\end{corollary}

Finally, we do not know of any state sum, skein relation, or deletion-contraction formula in graph theory that involves expanding along virtual crossings of a graph diagram as is done with the PK-bracket.  In knot theory, this idea will be used in a forthcoming paper by the second author on multi-virtual knot theory \cite{KPrep}. We speculate there may be other valuable graph theoretic formulae yet-to-be-discovered that also expand along virtual crossings.


\section{Ribbon graphs}\label{sec:ribbonGraphs}
In this section we introduce some preliminary notions of ribbon graphs which will be used throughout, but the reader should review \cite{BM-Color,BKR} and \cite[Section 1.1.4]{Moffat2013} for further details.  A {\em plane graph} $\Gamma$ is an embedding, $i:G \ra S^2$, of a connected planar graph $G$ into the sphere.  The key feature of plane graphs is that $S^2\setminus\iota (S^2)$ is a set of disjoint disks.  A ribbon graph captures this feature as well:  it is an embedding of a graph into a genus $g$ surface $\Sigma$ so that $\Sigma \setminus i(G)$ is a set of disks.   
\begin{definition}\label{Def:ribbongraph}
A {\em ribbon graph of a graph $G$} is an embedding $i:G\ra \Gamma$ where $G$ is thought of as a $1$-dimensional CW complex and $\Gamma$ is a surface with boundary where $\Gamma$ deformation retracts onto $i(G)$.  We say that \( G \) is the \emph{underlying graph} of \( \Gamma \), and that \( \Gamma \) is \emph{the surface associated to} the ribbon graph. 
\end{definition}

A drawing of a ribbon graph in the plane that respects the cyclic ordering of the edges at each vertex will be referred to as a \emph{ribbon diagram} (cf. \cite{BKR,BM-Color}).  We will often refer to the ribbon graph simply by $\Gamma$ and think of $\Gamma$ as a surface with an embedded graph $G$.  An orientation of a ribbon graph, if one exists, is an orientation of the surface. Let \( \overline{\Gamma} \) denote the closed smooth surface obtained by attaching discs to the boundary of $\Gamma$. 

\begin{definition}\label{definition:n-face-coloring}
An {\em $n$-face coloring of a ribbon graph $\Gamma$ (or $\overline{\Gamma}$)} is a choice of one of $n$ different colors (or more generally, labels) for each attaching disk of $\overline{\Gamma}$ such that no two disks adjacent to the same edge have the same color.
\end{definition}

For many computations in this paper we will need to choose a set of perfect matching edges, and when we are working in the context of a ribbon graph, we call the pair of a ribbon graph with a perfect matching a perfect matching graph.  

\begin{definition}\label{Def:pm}
A \emph{perfect matching} of an abstract graph $G(V,E)$ is a subset of the edges of the graph, $M\subset E$, such that each vertex is incident to exactly one edge in the subset. 
\end{definition}

\begin{definition}\label{Def:pm-graph}
A \emph{perfect matching graph}, denoted $\Gamma_M$, is a ribbon graph, $i:G\ra \Gamma$, together with a perfect matching $M$ of the graph $G$. We represent the perfect matching in a ribbon diagram of $\Gamma$ using thickened edges. 
\end{definition}

Throughout, an abstract graph $G(V,E)$ may be thought of as a connected $1$-dimensional CW complex by identifying vertices of $V$ with points and edges with segments that are glued to their coincident vertices.  Also, all graphs are multigraphs, which are allowed to have circles (edges with a single incident vertex) and multiple edges incident to the same two distinct vertices. Finally, ``vertex-free'' edges are allowed, i.e., circles.

The following construction will be useful for obtaining trivalent perfect matching graphs from a given, but not necessarily trivalent, ribbon graph, which is the blowup of a graph.  Blowups, even of trivalent graphs, come with a canonical perfect matching, which allows one to obtain graph and ribbon graph invariants.  

\begin{definition}
Let $G(V,E)$ be a graph and $\Gamma$ be a ribbon graph of $G$ represented by a ribbon diagram.  Define the {\em blowup of $\Gamma$}, denoted $\Gamma^\flat$, to be the ribbon diagram  given by replacing every vertex of $\Gamma$ with a circle as in
\begin{center}
\includegraphics[scale=0.07]{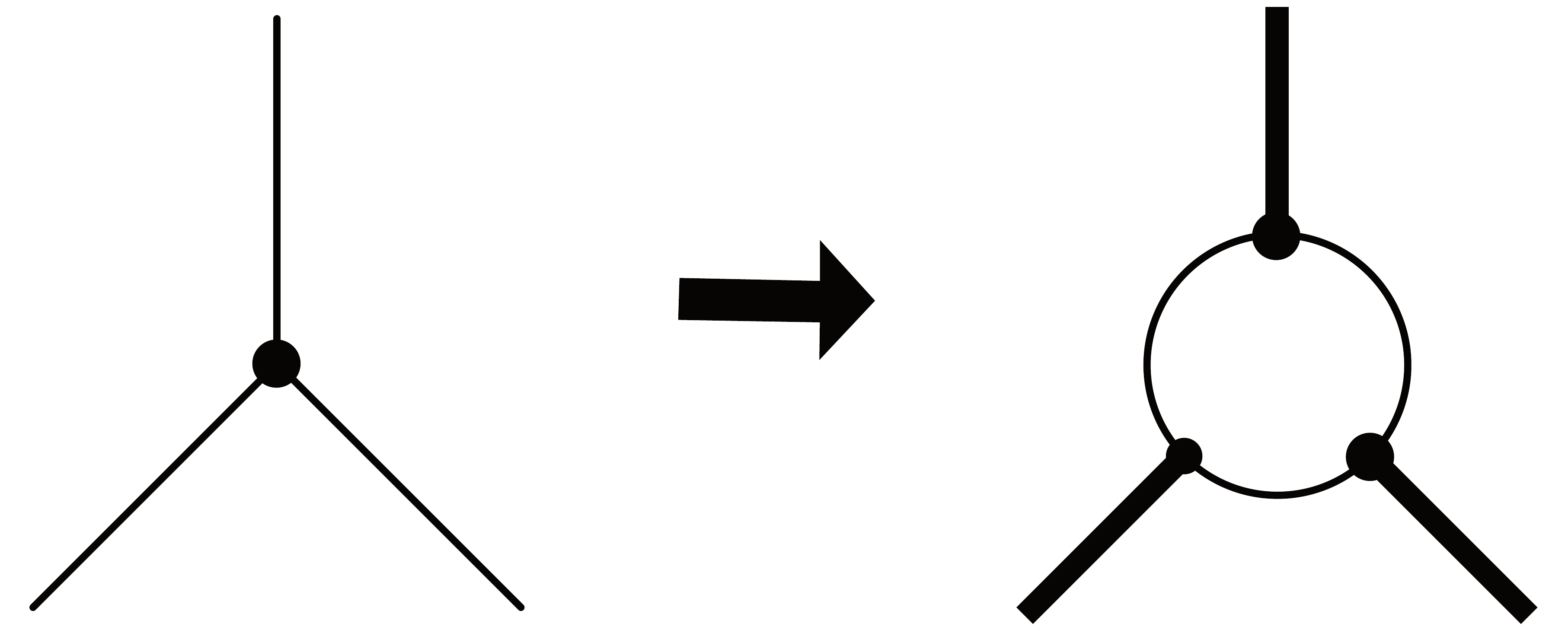}
\end{center}
A perfect matching can be associated to $\Gamma^\flat$  using the original edges $E$ of $\Gamma$ as shown in the picture above.  The  resulting perfect matching graph is $\Gamma^\flat_E$.  \label{def:blowup-of-a-graph}
\end{definition}


There is one additional type of n-coloring of a finite trivalent graph $G$ with perfect matching $M$ which we define for use in this paper.

\begin{definition}\label{def:PMnColor}
A {\em perfect matching $n$-coloring} of a trivalent graph $G$ with perfect matching $M$ is an assignment of colors to the non-matching edges of $G$ from the color set $\{1,2,...,n\}$ so that exactly two distinct colors are used to color the edges adjacent to each matching edge and these two colors both appear on edges at each end of the matching edge.
\end{definition}

Given a ribbon graph $\Gamma$ of a trivalent graph $G(V,E)$, the blowup will have the property that the circles of the all-zero smoothing (cf. \Cref{section:smoothing-states-hypercubes}) correspond to the faces of $\Gamma$.  In this case, the duality between face and edge colorings (cf. \Cref{fig:PMnColor}) implies that \Cref{def:PMnColor} and \Cref{definition:n-face-coloring} coincide.  If, however, one chooses a perfect matching $M\subset E$ instead of blowing up, a perfect matching $n$-coloring specifies a proper coloring for only the faces adjacent to a perfect matching edge of a ribbon graph.  In this case, the faces corresponding to the cycles of $G\setminus M$ are left uncolored.

\section{Filtered $n$-color homology and the total face color polynomial}\label{sec:filteredHom}
We first recall the essential constructions for filtered $n$-color homology that are needed to define the total face color polynomial (see \cite{BaldCohomology,BM-Color}). 

 \subsection{The hypercube of states}\label{section:smoothing-states-hypercubes} Let $G(V,E)$ be a trivalent graph and $M$ be a perfect matching of $G$. The number of vertices is even, and the number of perfect matching edges of $M$ is then $\ell=|V|/2$.  Label and order these edges  by $M=\{e_1, e_2, \dots, e_\ell\}.$ 
Let a perfect matching graph $\Gamma_M$ for $(G,M)$ be represented by a perfect matching diagram.  Resolve each perfect matching edge $e_i \in \Gamma_M$ in one of two possible ways according to two smoothings, that is, replace a neighborhood of each perfect matching edge $e_i$ in $\Gamma_M$ with $\IIDiag$, called a $0$-smoothing, or $\XDiag$, called a $1$-smoothing. The resulting set of immersed circles in the plane is called a {\em state} of $\Gamma_M$.

There are $2^\ell$ states of $\Gamma_M$, each of which can be indexed by an $\ell$-tuple of $0$'s and $1$'s that stand for the type of smoothing.  For  $\alpha= (\alpha_1, \dots,\alpha_\ell)$ in $\{0,1\}^\ell$, let $\Gamma_\alpha$ denote the state where each perfect matching edge $e_i$ has been resolved by an $\alpha_i$-smoothing. Let $|\alpha| = \sum_{i=1}^\ell \alpha_i$, and organize the states into columns based on the value of $|\alpha|$. The value of $|\alpha|$ will become the homological degree of the $n$-color theory.  

\subsection{Filtered $n$-color homology} We are now ready to associate vector spaces to the states of a perfect matching graph $\Gamma_M$ to build the chain complex for the filtered $n$-color homology.  We will only recall the necessary basics here and refer the reader to \cite{BM-Color} for more detail.  Let $k_\alpha$ be the number of immersed circles in the state $\Gamma_\alpha$, and associate the vector space $\large \widehat{V}_\alpha = \large \widehat{V}^{\otimes k_\alpha}$ to the state $\Gamma_\alpha$ where $\widehat{V}=\BC[x]/(x^n-1)$.

Define the complex $\widehat{C}^{*,*}(\Gamma_M)$ by

$$\widehat{C}^{i,*}(\Gamma_M)=\bigoplus_{\substack{\alpha\in\{0,1\}^n \\ i=|\alpha|}}\widehat{V}_\alpha.$$

To define the differential for the filtered $n$-color  homology, $\widehat{\del}: \widehat{C}^{i}(\Gamma_M) \rightarrow \widehat{C}^{i+1}(\Gamma_M)$, consider each edge $\Gamma_\alpha \rightarrow \Gamma_{\alpha'}$ in the hypercube and define a map, $\widehat{\del}_{\alpha\alpha'}:\widehat{V}_\alpha \ra \widehat{V}_{\alpha'}$ for $\widehat{V}_\alpha \subset \widehat{C}^i(\Gamma_M)$ and $\widehat{V}_\alpha' \subset \widehat{C}^{i+1}(\Gamma_M)$.  This map is determined by the change in the number of circles between $\Gamma_\alpha$ and $\Gamma_{\alpha'}$:  $m$ if two circles in $\Gamma_\alpha$ are merged into one, $\Delta$ if one circles splits into two, and $\eta$ if the number of circles is unchanged.  The differential can then be succinctly written using the local maps
\begin{eqnarray} \widehat{m}(x^i \ot x^j) &=&  x^{i+j}, \label{eq:wide-hat-differential-m}\\
\widehat{\Delta}(x^k) &=&  \sum_{\substack{0 \leq i,j < n \\ i+j \equiv (k + 2m) \!\!\!\!  \mod n}} x^i \ot x^j, \nonumber \label{eq:wide-hat-differential-delta}\\
\widehat{\eta}(x^k)&=& \sqrt n x^{k+m}, \nonumber \label{eq:wide-hat-differential-eta}
\end{eqnarray}
Here, $m= \frac{n}{2}$ if $n$ is even and $m= \frac{n-1}{2}$ otherwise.  We then define the {\em filtered $n$-color homology} to be (see Section 5.2 in \cite{BM-Color}):
\begin{equation}
\widehat{CH}_n^*(\Gamma_M,\mathbbm{C}):=H(\widehat{C}^{*,*}(\Gamma_M), \widehat{\del}).
\end{equation}

The basis $\{1,x,\ldots ,x^{n-1}\}$ is useful for thinking of the filtered $n$-color homology as the $E_\infty$ page of a spectral sequence whose $E_1$ page is the bigraded $n$-color homology (cf. \cite{BM-Color}).  For the purposes of this paper, it is advantageous to interpret the meaning of the elements in the vector space $\widehat{V}_\alpha$ for a state $\Gamma_\alpha$ using a different basis.  In this basis, the elements can be thought of as coloring the circles in the state $\Gamma_\alpha$. Each state can then be interpreted as coloring the circles with $n$ different colors. First, the definition:

\begin{definition} Let $n$ be a positive integer with $n>1$ and set $\lambda = e^{\frac{2\pi \mathrm{i}}{n}}$. The {\em color basis} of $\widehat{V}=\BC[x]/(x^n-1)$ is
$$c_i := \frac{1}{n}\left(1+ \lambda^i x +\lambda^{2i}x^2+ \lambda^{3i}x^3+\cdots+\lambda^{(n-1)i}x^{n-1}\right)$$
for $0\leq i \leq n-1$. \label{def:colorbasis}
\end{definition}

The $c_i$'s are the different colors of the theory.  Hence, when $n=4$, there are four colors $\{c_0, c_1, c_2, c_3\}$ for filtered $4$-color homology and so on. Also, note that choosing $\mathbbm{k}=\BC$ is now advantageous to make the $c_i$'s well-defined for $n>2$ since $\lambda$ is an $n$th root of unity.

\begin{lemma}[cf. Lemma 5.9 in \cite{BM-Color}]
\label{lem:widehat-maps}
In the color basis, the following equations hold:
\begin{enumerate}
\item $c_i \cdot c_j =  \delta^{ij} c_j,$ hence $\widehat{m}(c_i \ot c_j) =  \delta^{ij} c_j$,
\item $\widehat{\Delta}(c_i) =  n \lambda^{-2mi} c_i \ot c_i,$
\item $\widehat{\eta}(c_i) = \sqrt n \lambda^{-mi} c_i$, and
\item $(\lambda^i x) \cdot c_i =  c_i$.
\end{enumerate}
\end{lemma}

As shown in \cite{BM-Color} the main advantage of the color basis is that it allows one to conceptualize the elements of filtered $n$-color homology as proper colorings (cf. Sections 6 and 7 of \cite{BM-Color}). In particular, the Color Basis Lemma (cf. Lemma 6.4 of \cite{BM-Color}) implies that no two distinct colorings $c_I, c_J\in \widehat{V}_\alpha$ map to the same coloring $c'_I$ of $\widehat{V}_\alpha'$ or vice versa. More specifically, if $\widehat{\del}_{\alpha\alpha'}:\widehat{V}_\alpha \ra \widehat{V}_{\alpha'}$ is the edge-differential ($\widehat{m}$, $\widehat{\Delta}$, $\widehat{\eta}$) corresponding to an edge in the hypercube of states of $\Gamma_M$ from $\Gamma_\alpha$ to $\Gamma_{\alpha'}$ (and $\widehat{\del}_{\alpha \alpha'}^*$ is defined similarly) then the maps $\widehat{\del}_{\alpha\alpha'}$ and $\widehat{\del}^*_{\alpha\alpha'}$ are one-to-one on color basis elements that are not in their kernels.  This turns out to be the key to showing that the homology classes are supported individual states, which is discussed below.  


\subsection{The harmonic colorings of a state} Next, we recall (again from \cite{BM-Color}) the harmonic colorings of a state, $\widehat{\mathcal{CH}}_n(\Gamma_\alpha)$, which can be thought of as the harmonic elements of a Dirac-like operator that exist only on the state $\Gamma_\alpha$. This subspace of $\widehat{\mathcal{CH}}_n^i(\Gamma)$ is the  harmonic elements of the vector space $\widehat{V}_\alpha$ that do not depend on elements of other state vector spaces in $\widehat{C}^i(\Gamma)=\oplus_{|\alpha| = i}\widehat{V}_\alpha$ to form a harmonic class in $\widehat{\mathcal{CH}}_n^i(\Gamma)$. 

Let $\Gamma_\alpha$ be a state of the hypercube for perfect matching graph $\Gamma_M$.  Consider all states $\Gamma_{\alpha'}$ such that $|\alpha'|=|\alpha|+1$ where there is an edge in the hypercube between $\Gamma_\alpha$ and $\Gamma_{\alpha'}$.  Denote the union of these states by $\Gamma_\alpha^+ = \cup \Gamma_{\alpha'}$.  Then $\widehat{C}^{i+1}(\Gamma^+_\alpha) \subset \widehat{C}^{i+1}(\Gamma)$ is made up of the direct sum of vector spaces $\oplus V_{\alpha'}$. The restriction of the metric (cf. Section 4 of \cite{BM-Color}) to this subspace remains a metric.

Similarly, define $\widehat{C}^{i-1}(\Gamma_\alpha^-) \subset \widehat{C}^{i-1}(\Gamma)$ consisting of all vector spaces $\widehat{V}_{\gamma}$ such that there is an edge from $\Gamma_\gamma$ to $\Gamma_\alpha$ in the hypercube of states. 

Define $\widehat{\del}_\alpha:\widehat{V}_\alpha \ra \widehat{C}^{i+1}(\Gamma^+_\alpha)$ by taking the sum of all  differentials from  $\widehat{V}_\alpha$ to the $(|\alpha|+1)$-states. Similarly, define  $\widehat{\del}_\alpha^*: \widehat{V}_\alpha \ra \widehat{C}^{i-1}(\Gamma^-_\alpha)$ to be the sum of all nontrivial adjoint maps from $\widehat{V}_\alpha$ to $(|\alpha|-1)$-states.

\begin{definition}\label{defn:harmonic-coloring-of-a-state}
The {\em harmonic colorings of a state $\Gamma_\alpha$}, denoted $\widehat{\mathcal{CH}}_n(\Gamma_\alpha)$, is the set of elements of $\widehat{V}_\alpha$ that is in the kernel of $\widehat{\del}_\alpha$ and the kernel of $\widehat{\del}^*_\alpha$.  That is,
$$\widehat{\mathcal{CH}}_n(\Gamma_\alpha)=\ker \widehat{\del}_\alpha \bigcap \ker \widehat{\del}^*_\alpha.$$
\end{definition}

\begin{figure}[H]
\includegraphics[scale=.8]{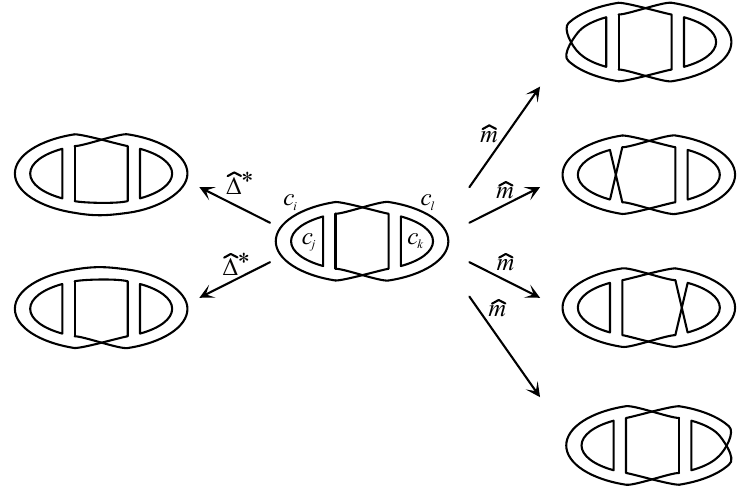}
\caption{A harmonic coloring when $c_i, c_j, c_k$ and $c_l$ are all distinct.}
\label{fig:harmonic}
\end{figure}
It is clear from the definition of the local differentials and \Cref{lem:widehat-maps} that elements of the kernel of $\widehat{\del}_\alpha$ must look like \Cref{fig:harmonic}, in that only multiplications emanate from the state ($\widehat{\Delta}$ and $\widehat{\eta}$ maps have trivial kernel).  For the local adjoint maps we have the following definitions (cf. Lemma 6.2 of \cite{BM-Color}):
\begin{eqnarray*}\label{eqn:adjoint-of-m-delta-eta}
\widehat{m}^*(c_i) &=& c_i\ot c_i\\ \nonumber
\widehat{\Delta}^*(c_i\ot c_j) &=& n \lambda^{2mi}\delta^{ij}c_i\\ \nonumber
\widehat{\eta}^*(c_i) &=& \sqrt{n} \lambda^{mi}c_i.\nonumber
\end{eqnarray*}
Again, it is clear from the local differentials that elements of the kernel of $\widehat{\del}_\alpha^*$ must look like \Cref{fig:harmonic}, in that only $\widehat{\Delta}^*$ maps emanate from the state ($\widehat{m}^*$ and $\widehat{\eta}^*$ maps have trivial kernel).  While such states are the only ones that can support colorings, more is shown in Theorem D of \cite{BM-Color}.  In particular, it is shown that such harmonic colorings generate the filtered $n$-color homology.  While the theorem is stated for the blowup of the graph (i.e. $\Gamma_E^\flat$) in \cite{BM-Color}, the proof given there also works for any perfect matching graph, $\Gamma_M$. Thus, we conclude the following new theorem:
\begin{theorem}[cf. Theorem D in \cite{BM-Color}]
Let $\Gamma_M$ be a perfect matching graph of an abstract graph $G(V,E)$ with perfect matching $M\subset E$.  Then the filtered $n$-color homology is generated by harmonic colorings, i.e. 
$$\widehat{CH}_n^i(\Gamma_M,\mathbb{C}) \cong \bigoplus_{|\alpha|=i} \widehat{\mathcal{CH}}_n(\Gamma_\alpha).$$
Moreover, the harmonic colorings correspond to colorings, and we obtain that the dimension of $\widehat{\mathcal{CH}}_n(\Gamma_\alpha)$ is equal to  the number of perfect matching $n$-colorings of $\Gamma_\alpha$, that is, the number of proper face colorings of $\overline{\Gamma}_\alpha$ in which the faces that correspond to the cycles of $G\setminus M$ are left uncolored.
\label{thm:harmonicsGenerate}
\end{theorem}

It is also shown in \cite{BM-Color} that the Euler characteristic of this homology is the evaluation of the usual Penrose polynomial found in the literature evaluated at $n$.  However, taking the Poincar\'{e} polynomial of this homology yields another invariant of the perfect matching graph.

\begin{definition} Let $G(V,E)$ be a connected trivalent graph, $M\subset E$ a perfect matching, and let $\Gamma_M$ be any perfect matching graph of $(G,M)$. The Poincar\'{e} polynomials of the filtered $n$-color homologies generate the {\em $2$-variable total face color polynomial} which is characterized by
$$T(\Gamma_M,n,t) := \sum_{|\alpha|=i} t^i \dim \widehat{\mathcal{CH}}_n(\Gamma_\alpha) $$
when evaluated at $n\in\BN$. The {\em total face color polynomial of $\Gamma_M$} is $T(\Gamma_M,n):=T(\Gamma_M,n,1)$.   Finally, define the \emph{total face color polynomial of $\Gamma$} to be the total face color polynomial of the blowup, $T(\Gamma,n,t) := T(\Gamma_E^\flat,n,t)$ and $T(\Gamma,n) := T(\Gamma,n,1)$.
\label{definition:totalfacecolorpolynomial}
\end{definition}

The definition of the $2$-variable total face color polynomial given in \cite{BM-Color} is equivalent to $T(\Gamma^\flat_E,n,t)$ above, where $\Gamma$ is a ribbon diagram of an abstract graph $G(V,E)$.  This indeed gives a polynomial that counts the total number of face colorings of the ribbon graphs in the hypercube of states when evaluated at $t=1$.  If $G$ is trivalent, it is an invariant of the abstract graph, not just the ribbon graph used (cf. Section 7 of \cite{BM-Color}). Therefore we define $T(G,n):=T(\Gamma,n)$ for any ribbon graph $\Gamma$ of a trivalent graph $G$.

More generally, if $G$ is trivalent and equipped with a perfect matching $M$, then $T(\Gamma_M,n,t)$, is an invariant of the perfect matching graph $\Gamma_M$ (i.e. it depends on both the ribbon graph, and the chosen perfect matching).  Thus, we may conceive of the hypercube of smoothings as ribbon graphs in which the faces that correspond to the cycles of $G\setminus M$ are not colored.  Colorings of such ribbon graphs are equivalent to perfect matching $n$-colorings (cf. \Cref{def:PMnColor} and \Cref{fig:PMnColor}). 




\section{The Penrose and Penrose-Kauffman Brackets}

In this section we define the Penrose-Kauffman bracket, or PK-bracket, which is a coloring polynomial in the variable $n$ (that can be taken to be a positive integer), defined for trivalent graphs $G$ with perfect matching $M.$ This polynomial is a generalization of the $PK$ evaluation at $n=3$ studied in \cite{Kauffman}. The special case at $n=3$ counts the number of $3$-edge colorings of an arbitrary
trivalent graph $G$ (no perfect matching required) via a generalization of the original Penrose evaluation \cite{Kauffman, Penrose}. 

The key point about the evaluation of the PK-bracket is that it
gives the total number of colorings of the graph $G$ for a ribbon diagram of $G$ in the plane, and it follows the original Penrose expansion, with an extra caveat for the singularities
of the immersion. In our generalization, we will follow the same procedure for the PK-bracket and obtain a count of special colorings of the perfect matching graph 
using $n$ colors. The Penrose-Kauffman bracket extends the Penrose evaluation to arbitrary trivalent graphs with perfect matchings following the methods described in \cite{BaldCohomology,BKR,BLM,BM-Color,Kauffman}.\\

We begin with a brief description of the Penrose polynomial, then introduce a bracket for counting perfect matching $n$-colorings, and lastly introduce the PK-bracket.

\subsection{The Penrose Polynomial} 
In 1971, Roger Penrose \cite{Penrose} described several formulas for computing the number of $3$-edge colorings of a planar trivalent graph, one of which led to his famous polynomial.  We now recall an intuitive definition of the Penrose polynomial from \cite{BM-Color} that is defined using brackets (see also \cite{BaldCohomology}).

\begin{definition}\label{defn:penrose_poly}
Let $G$ be a trivalent graph with a perfect matching $M$, and let $\Gamma_M$ be a perfect matching graph for the pair $(G,M)$.  Then the {\em Penrose polynomial}, denoted $\left[\Gamma_M\right]_n$,  is found by recursively applying the bracket $$\left[ \PMEdgeDiag \right]_{\! n} = \left[ \IIDiag  \right]_{\! n}  - \left[ \XDiag \right]_{\! n}$$ to perfect matching edges of $\Gamma_M$ and setting the value of immersed loops to $\left[ \bigcirc \right]_{\! n} = n$.  
\end{definition}

Penrose observed that evaluation of the polynomial at $n=3$ computes the number of $3$-edge colorings for planar graphs.  

The minus sign appearing in the recursive relation makes the Penrose polynomial amenable to categorification, and in \cite{BM-Color} its evaluation at $n$ was shown to be the Euler characteristic of the bigraded $n$-color homology, which via a spectral sequence ties the Penrose polynomial to the filtered $n$-color homology and the total face color polynomial.

\subsection{A bracket that counts perfect matching $n$-colorings.}\label{sec:colorBracket}
We point out first an intermediary, purely combinatorial interpretation of the coloring count for $(G,M).$ Define the bracket, denoted by  $\{G,M \}$, by the recursion
$$\left\{ \PMEdgeDiag  \right\} = \left\{ \VGlyph \right\} + \left\{  \CGlyph  \right\},$$
where it is understood that $$\{ G,M\} = \sum_{S} \{S\},$$ where each matching edge has been replaced by the glyphs in the recursion above to form a collection of 
state configurations consisting in circles connected by the wiggly glyphs in the form
$ \{ \VGlyph \} $ and $\{  \CGlyph  \}.$  The evaluation of a state $S$ is defined to be equal to the number of ways to color the circles in $S$ with $n$ colors so that each pair of arcs joined by a wiggly glyph are colored differently. \\

\begin{example}
Consider the theta graph below.  Observe that after resolving the matching edge, the cross-resolution cannot be colored with different colors at the wiggly glyph.

\begin{eqnarray*}
\left\{ \ThetaG \right\} &=& \left\{ \ThetaZeroSquiggle \right\} + \left\{ \ThetaOneSquiggle \right\}\\
&=& n(n-1) + 0.
\end{eqnarray*}
\end{example}

Since, by its definition, $\{G,M\}$ counts those colorings of the perfect matching graph so that exactly two distinct colors appear at each matching edge satisfying our conditions for an $n$-coloring of $(G,M),$ it follows that $\{G,M\}$ is equal to the number of perfect matching $n$-colorings.  Moreover, \Cref{thm:harmonicsGenerate} states that the total face color polynomial gives the same count.  Thus we obtain the following theorem.

\begin{theorem}
Let $\Gamma_M$ be a perfect matching graph for the pair $(G,M)$.  Then
$$T(\Gamma_M,n) = \left\{G,M \right\}.$$
\label{thm:TEqualsMatching}
\end{theorem}

Note that $\{G,M\}$ is defined independent of any planar immersion of the graph $G$, but it is complicated to calculate directly, since each of the $2^{|M|}$ states has to be considered individually as a separate coloring problem.  This makes it useful from a theoretic point of view, but not necessarily for calculation.\\

\subsection{The Penrose-Kauffman Bracket}\label{sec:PKBP}

The Penrose-Kauffman bracket provides a way to modify the Penrose polynomial so that one still obtains counts of $3$-edge colorings for nonplanar graphs for $n=3$.  
 
\begin{definition}\label{defn:PenroseKauffman}
Let $G$ be a trivalent graph with a perfect matching $M$, and let $\Gamma_M$ be a perfect matching graph for the pair $(G,M)$. Then the {\em Penrose-Kauffman bracket} (or PK-bracket), denoted $\left\llbracket \Gamma_M \right\rrbracket$,  is found recursively by applying the relations 
\begin{eqnarray}
\left\llbracket \PMEdgeDiag \right\rrbracket = \left\llbracket \IIDiag  \right\rrbracket  - \left\llbracket \XDiag \right\rrbracket  \hspace{2 cm} \text{(Penrose)} \label{eqn:Penrose}\\
\left\llbracket \SquareVirtual \right\rrbracket = 2 \left\llbracket \NodeVirtual  \right\rrbracket  - \left\llbracket \XDiag \right\rrbracket \hspace{1.45 cm} \text{(Kauffman)}\label{eqn:Kauffman}
\end{eqnarray}
and $\left\llbracket \bigcirc \right\rrbracket= n$ to ribbon graph $\Gamma_M$, where the node \NodeVirtual, means that the two arcs are treated as one circle.  
\end{definition}

These recursions mean that the evaluation takes the form of a Penrose expansion except that we keep track of the original immersed crossings, denoted by $\SquareVirtual$ and then the last relation
expands further each original immersed crossing in terms of an ordinary crossing in the expansion, $\XDiag$, and a {\it fused} crossing in the form $\NodeVirtual.$ If two circles are joined at a fused crossing
then they together contribute $n,$ the same as a single circle. In general, a complex of circles connected by fusions contributes only $n.$ 

Note that the virtual crossing in \Cref{eqn:Kauffman} involve only non-perfect matching edges.  However, for an arbitrary immersion of a perfect matching graph, virtual crossings may involve one or more perfect matching edges.  If they do, however, one may always modify the immersion to produce an equivalent perfect matching graph in which virtual crossings always avoid the perfect matching edges.

\begin{lemma}
Let $G(V,E)$ be an abstract graph and $M\subset E$ a perfect matching.  Any ribbon diagram of a perfect matching graph is equivalent (as a ribbon graph) to one in which the virtual crossings involve only the edges of $G\setminus M$.  
\end{lemma}

\begin{proof}
Beginning with the ribbon diagram, contract all of the perfect matching edges to points so that the resulting immersion involves only edges of $G\setminus M$.  Then expand the contracted edges by a small amount to produce the desired immersion.  
\end{proof}

We now present the following lemma which will be useful for computation in the examples to follow.

\begin{lemma}
The PK-bracket satisfies the following relations.
\begin{eqnarray*}
\left\llbracket \raisebox{-0.42\height}{\includegraphics[scale=.3]{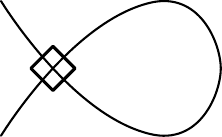}} \right\rrbracket & = & \left\llbracket \raisebox{-0.42\height}{\includegraphics[scale=.3]{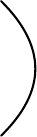}} \right\rrbracket, \\
\left\llbracket \raisebox{-0.42\height}{\includegraphics[scale=.3]{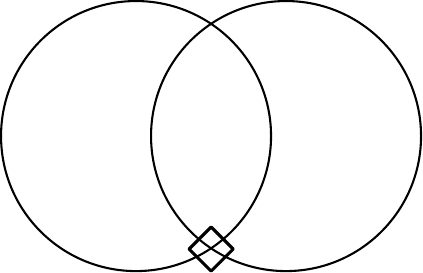}} \right\rrbracket & = & 2n-n^2.
\end{eqnarray*}
\label{prop:relations}
\end{lemma}

\begin{proof}
Notice that in the calculation below, the node occurs on a single arc, which must be given a single color regardless. 
\begin{eqnarray*}
\left\llbracket \raisebox{-0.42\height}{\includegraphics[scale=.3]{SVR1.pdf}} \right\rrbracket &=& 2\left\llbracket \raisebox{-0.42\height}{\includegraphics[scale=.3]{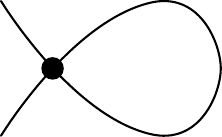}} \right\rrbracket - \left\llbracket \raisebox{-0.42\height}{\includegraphics[scale=.3]{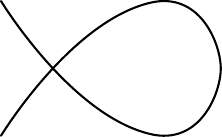}} \right\rrbracket.
\end{eqnarray*}
For the second relation, observe that the two circles are treated as one when the square virtual crossing is replaced with a node, but as two circles when it is treated as an ordinary virtual crossing.
\end{proof}


\begin{example}\label{ex:doubleTheta}
For the double theta graph below the PK-bracket and the Penrose polynomial are equal since there are no virtual crossings.  Observe that for $n=3$ the polynomial evaluates to 12, which counts the number of $3$-edge colorings of the graph.
\begin{eqnarray*}
\left\llbracket \DoubleTheta \right\rrbracket &=& \left\llbracket \DoubleThetaZero \right\rrbracket - \left\llbracket \DoubleThetaOne \right\rrbracket\\
&=& (n-1) \left\llbracket \ThetaG\right\rrbracket\\
&=& (n-1) \left(\left\llbracket \ThetaZero\right\rrbracket - \left\llbracket \ThetaOne\right\rrbracket\right)\\
&=& n(n-1)^2.
\end{eqnarray*}
The Penrose polynomial, and similarly the PK-bracket, depend on the choice of perfect matching, as can be seen when one computes the polynomial for the matching shown below.
\begin{eqnarray*}
\left\llbracket \DoubleThetaB \right\rrbracket &=& \left\llbracket \DoubleThetaBOne \right\rrbracket - \left\llbracket \DoubleThetaBTwo \right\rrbracket\\
&=& n(n-1) -  \left\llbracket \DoubleThetaBThree \right\rrbracket + \left\llbracket \DoubleThetaBFour \right\rrbracket\\
&=& n(n-1) - n + n^2\\
&=& 2n(n-1).
\end{eqnarray*}
While the Penrose polynomial and PK-bracket both depend on the choice of perfect matching, the evaluation of the polynomial at $n=3$ does not (cf. \cite{BM-Color,Kauffman,Penrose}).  If one wishes to obtain a polynomial that is invariant of the choice of perfect matching, one may work with the blowup with its canonical perfect matching.
\end{example}

\begin{example}\label{ex:K33}
For the $K_{3,3}$ graph, we use \Cref{prop:relations} to observe that the PK-bracket is the same as that of the double theta graph of \Cref{ex:doubleTheta}.

\begin{eqnarray*}
\left\llbracket \KThreeThree \right\rrbracket &=& \left\llbracket \KThreeThreeOne \right\rrbracket - \left\llbracket \KThreeThreeTwo \right\rrbracket\\
&=& \left\llbracket \DoubleTheta \right\rrbracket - \left\llbracket \KThreeThreeThree \right\rrbracket + \left\llbracket \KThreeThreeFour\right\rrbracket\\
&=& \left\llbracket \DoubleTheta \right\rrbracket - \left\llbracket \KThreeThreeFive \right\rrbracket + \left\llbracket \KThreeThreeSix \right\rrbracket\\
& \ \  & \ \ \ \ \ \ \ \ \ \ \ \ \ \ \ \ \ \ \ \ \ \ \ \    + \left\llbracket \KThreeThreeSeven \right\rrbracket - \left\llbracket \KThreeThreeEight \right\rrbracket\\
& = & n(n-1)^2 - n^2 + n + n - (2n -n^2)\\
& = & n(n-1)^2
\end{eqnarray*}

Notice that if we wish to calculate the Penrose polynomial, the only change above is that the final diagram contributes $-n^2$ instead of $2n-n^2$.  After simplifying, we see that the Penrose polynomial satisfies
$$ \left[ \KThreeThree \right]_n = n(n-1)(n-3).$$
Comparing the two polynomials, we see that the Penrose polynomial evaluates to $0$ at $n=3$, but the PK-bracket evaluates to 12, which as one may check, is the number of $3$-edge colorings of the graph.

\end{example}

\begin{example}\label{ex:Petersen}
For the Petersen graph $Pet(V,E)$ in \Cref{P} one may calculate the PK-bracket using the perfect matching indicated, and find that it evaluates to $0$.\footnote{A mathematica program is included in the appendix to do the computation of the PK-bracket.}  However, if one computes the PK-bracket for the blowup one finds that
$$\left \llbracket {Pet}_E^\flat \right \rrbracket = (n-4)(n-3)(n-2)(n-1)n(40+2n).$$

\end{example}


In each example, we observe that the PK-bracket is the same as the total face color polynomial, which lead the authors to the discovery that the PK-bracket and the total face color polynomial are the same:

\begin{theorem}  Let $G$ be a trivalent graph with a perfect matching $M$, and let $\Gamma_M$ be a perfect matching graph for the pair $(G,M)$. Then,
$$\left\llbracket \Gamma_M \right\rrbracket = \{ G,M\}.$$
\label{thm:PKEqualsMatching}
\end{theorem}

\begin{proof} Define two diagrammatic tensors as shown below.
The indices run in the set $\{1,2,\cdots n\}$ for $n$ colors.\\
\begin{equation}
\BigYMTens = \left\{ \begin{array}{ll} \ \ 1  \ \ \ \ & \text{if } a=c, b=d, a \ne b,\\[.3cm]
-1  & \text{if } a = d, b = c, a \ne b,\\[.3cm]
0 & \text{otherwise.}
\end{array}\right.
\end{equation}

\begin{equation}
\BigIMTens = \left\{ \begin{array}{ll} \ \ 1  \ \ \ \ & \text{if } a=d, b=c, a = b,\\[.3cm]
-1  & \text{if } a = d, b = c, a \ne b,\\[.3cm]
0 & \text{otherwise.}
\end{array}\right.
\label{eqn:SquareVirtual}
\end{equation}

Note that with $\delta^{a}_{b}$ denoting a Kronecker delta, we have the formula
$$\YMTens = \delta^{a}_{c}\delta^{b}_{d} - \delta^{a}_{d}\delta^{b}_{c}$$
and that 
$$\IMTens = 2 \FTens - \CTens = 2 \FTens - \delta^{a}_{d}\delta^{b}_{c},$$
where
$\FTens$ is equal to $1$ only when $a=b=c=d$ and is $0$ otherwise.\\

Define $[G,M]$ as the tensor contraction of $(G,M)$ with respect to these tensors in the sense of Penrose \cite{Penrose}. That is,
$[G,M]$ equals the sum over all possible index assignments to the {\it non-matching edges} of $G$ where we take the product of tensor values for each assignment of the
indices. It follows from the tensor definitions that in order for an index assignment to contribute to the summation, it must be a coloring of a state $S$ of the color bracket for $(G,M).$  
The contribution is (by the above assignments) equal to $(-1)^{A + B}$ where $A$ is the number of crossed glyph, $\CGlyph$, contributions, and $B$ is the number of immersion tensors, $\IMTens$, with 
$a\ne b.$ By the Jordan Curve Theorem (since the graphs are immersed in the plane), $A+ B$ is even, and hence each state contributes $+1$ to the summation. This proves that 
$[G,M] = \{ G,M\}.$ On the other hand, it follows from the tensor definitions that the relations of the PK-bracket, $\left\llbracket \PMEdgeDiag \right\rrbracket = \left\llbracket \IIDiag\right\rrbracket - \left\llbracket\XDiag \right\rrbracket$ and $\left\llbracket \SquareVirtual \right\rrbracket = 2\left\llbracket \NodeVirtual \right\rrbracket - \left\llbracket\XDiag \right\rrbracket,$ are respected by the tensors as well. Thus
$\left\{G,M\right\} = \left\llbracket \Gamma_M \right\rrbracket.$ This completes the proof.

\end{proof}

\section{Concluding Remarks on Snarks}

A  {\it snark} \cite{Gardner} is a trivalent graph that is not properly $3$-edge colorable. The Petersen graph \cite{Petersen} (See Figure~\ref{P}) is a fundamental example of a nonplanar snark.
Tutte conjectured \cite{Tutte} that every nonplanar snark $G$ has a Petersen minor (i.e. that the Petersen graph can be obtained from $G$ by operations of contraction and
deletion). Tutte's conjecture is still open with results in its favor by Robertson, Seymour and Thomas \cite{RST}. Rufus Isaacs
wrote a key paper \cite{Isaacs} showing how to construct infinite families of nonplanar snarks. \\
\begin{figure}[htb]
     \begin{center}
     \begin{tabular}{c}
     \includegraphics[width=8cm]{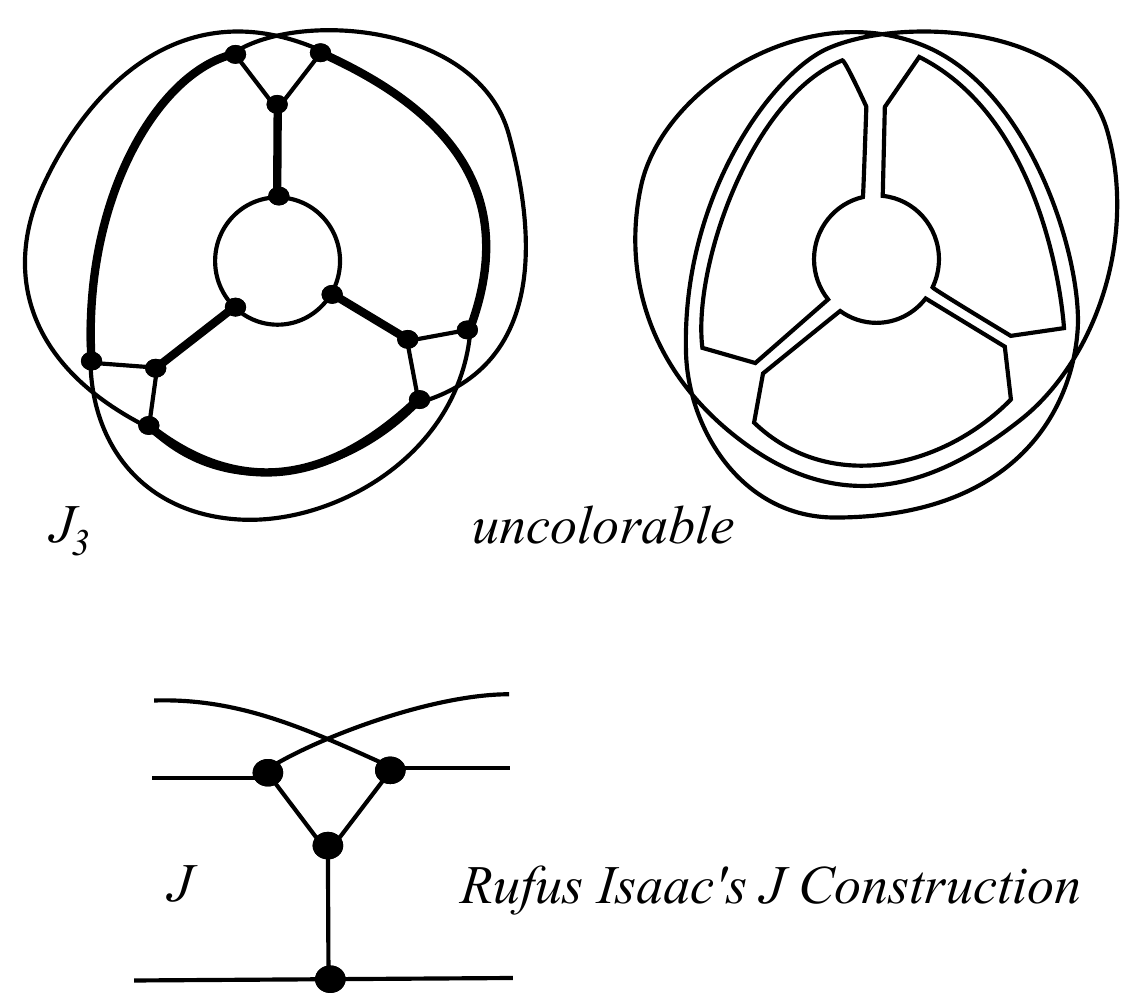}
     \end{tabular}
     \caption{ Isaacs Construction and $J_{3}.$}
     \label{J3}
\end{center}
\end{figure}

\begin{figure}[htb]
     \begin{center}
     \begin{tabular}{c}
     \includegraphics[width=8cm]{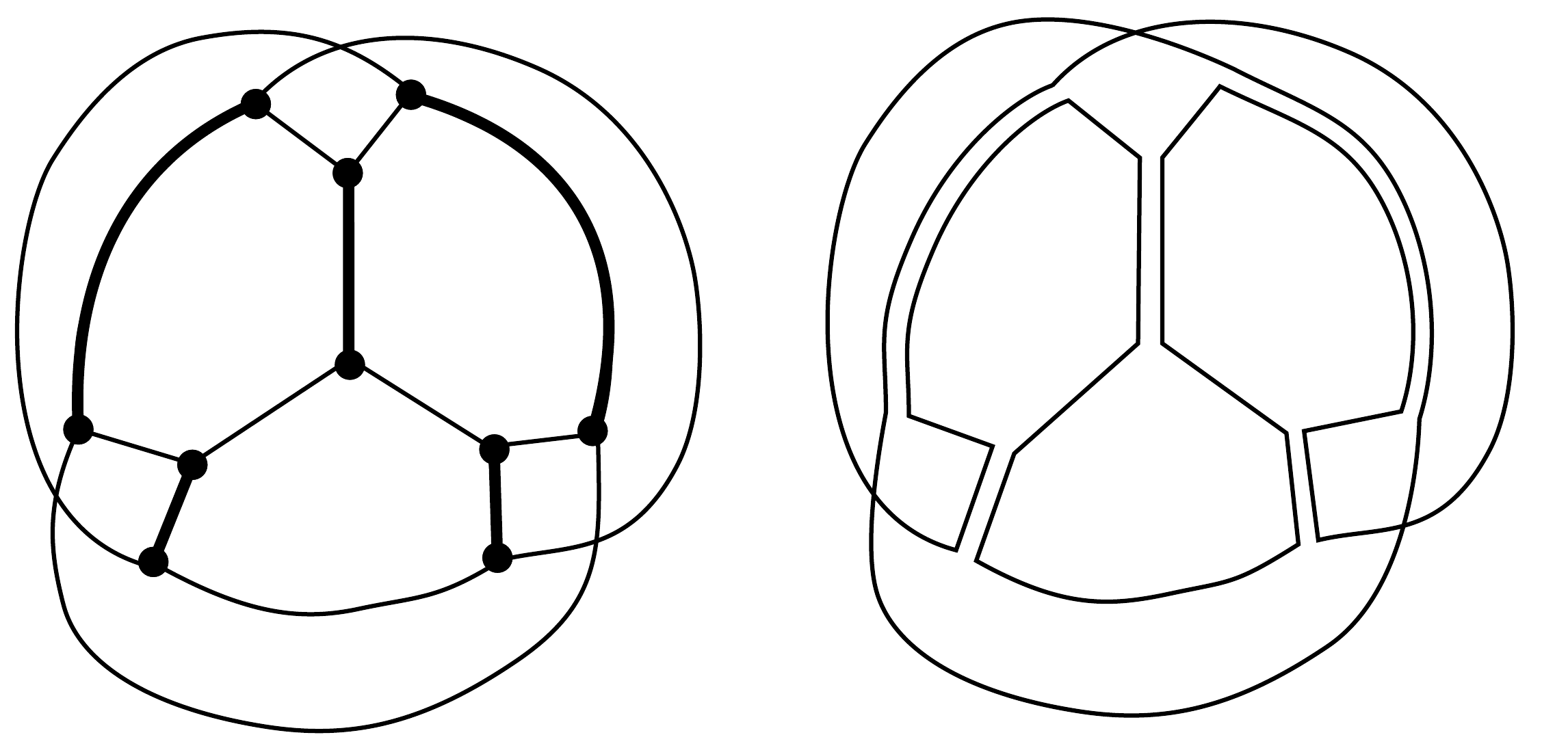}
     \end{tabular}
     \caption{The Petersen Graph, and its all-zero state. }
     \label{P}
\end{center}
\end{figure}


Isaacs' construction is based on the graph labeled $J$ in Figure~\ref{J3}. The ``circuit element" $J$
can be regarded as a box with three inputs and three outputs. A circular interconnection of $m$ copies of $J$ is denoted as Isaacs' $J_m.$ For $m$ odd it is not hard to prove 
that $J_m$ is not $3$-edge colorable. The $J_3$ snark can be contracted to the Petersen graph as shown in Figure~\ref{P}. We choose, for the sake of making example calculations, a perfect matching $M$ on the $J_{m}$ as shown in Figure~\ref{J3}. The figure illustrates the matching for $J_3,$ and it should be clear to the reader how to extend it to $J_m.$ 

We find that 
$\left\llbracket Pet_M \right\rrbracket = 0$ for any perfect matching graph $Pet_M$ for the Petersen graph $Pet(V,E)$ with perfect matching $M\subset E$.  On the other hand, if we take a perfect matching graph for $J_3$ with perfect matching $M$, we find
that $\left\llbracket (J_3,M) \right\rrbracket = n(-6 + 11 n - 6 n^2 + n^3 ).$  Note that this evaluates to $0$ for $n=3$ and to $24$ for $n=4.$ This means that there are perfect matching $4$-colorings of $J_3$ (alternatively, $4$-face colorings that leave the faces corresponding to the cycles of $J_3\setminus M$ uncolored) with the perfect matching
shown in Figure~\ref{J3}. In fact, this figure shows a state of $(J_3,M)$ with four mutually touching loops. This state can be colored in $4! = 24$ ways, and so we conclude that the $24$
perfect matching $4$-colorings come from this very state. A similar argument applies to $J_{m}$ for $m$ odd and perfect matching $M$ generalizing the choice in Figure~\ref{J3}. In the generalization, the corresponding state has $m+1$ loops, none of them self-touching. From this it follows that the polynomial $\left\llbracket  (J_m,M)  \right\rrbracket$ is non-zero for $m$ odd and greater than one. Note that  all states of the Petersen graph (with respect to our chosen perfect matching in Figure~\ref{P}) have self-touching loops. This explains why the polynomial for the Petersen graph vanishes. There can be no colorings
of it for any number of colors.\\

There are many questions that arise about these generalized coloring polynomials. So far, we have only seen the Petersen graph (as a non-trivial snark) receive the polynomial equal to zero.  We have just pointed out that all the Issacs $J_m$ will, with appropriate perfect matchings, have non-zero polynomials. Our calculations have shown that there also exists a perfect matching on $J_3$ (different from \Cref{J3}) so that the total face color polynomial is zero.  Therefore one may ask the following:
 
 \question \label{q:zero} Does there exist a perfect matching on $J_m$, for $m$ odd and $m>3$, so that the total face color polynomial is zero? More generally, when does a non-trivial snark have zero total face color polynomial for some perfect matching?

\remark If one passes to the blowup, the total face color polynomial is nontrivial if and only if the graph has a cycle double cover.  \Cref{q:zero} is about whether the polynomial can be zero when one does not first blow up and instead works with a perfect matching.


\appendix
\section{Mathematica Code}
We present below Mathematica code that may be used to calculate the total face color polynomial.  The code requires one to input a trivalent graph in a form of planar diagram, or ``graph PD'' notation, which we describe here.

Given a perfect matching diagram $\Gamma$ of a trivalent graph $G$ and let $M$ be a perfect matching for $G$.  We can create a perfect matching diagram of $\Gamma$ by drawing it in the plane and marking the perfect matching edges.  The ribbon structure of $\Gamma$ is encoded by the cyclic ordering of the edges at each vertex, and the ribbons are assumed to lay flat on the plane.  We then number the edges of $G\setminus M$ consecutively.  The starting point is arbitrary, and once a cycle closes up, we continue numbering any remaining cycles consecutively as well.  The perfect matching edges are left unadorned (see \Cref{fig:PetJ3}).    

\begin{figure}[htb]
     \begin{center}
     \begin{tabular}{c}
     \includegraphics[scale=.65]{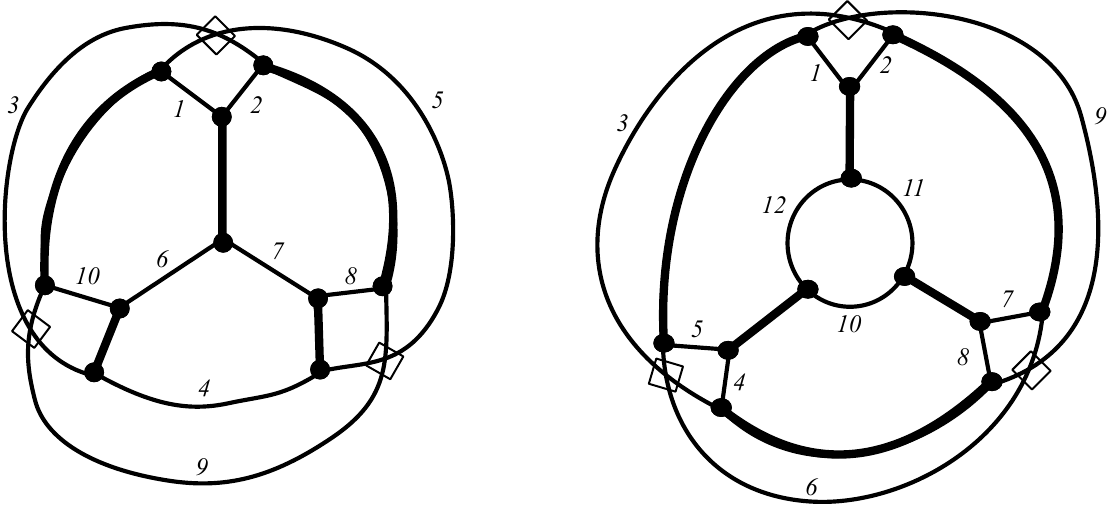}
     \end{tabular}
     \caption{The Petersen graph and $J_3$, with labeled arcs. }
     \label{fig:PetJ3}
\end{center}
\end{figure}

Note that if each cycle of $G\setminus M$ has at least $2$ edges, then the labeling induces an orientation on the cycles of $G\setminus M$.  We can assign a $4$-tuple to each matching edge by beginning with the incoming edge (following the orientation) and going counterclockwise or clockwise around the matching edge as shown (we ignore the orientation on the edges of $G\setminus M$ labeled c and d  in \Cref{fig:PDCode}).

\begin{figure}[htb]
     \begin{center}
     \begin{tabular}{c}
     \includegraphics[scale=.75]{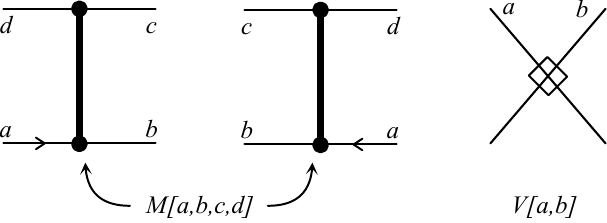}
     \end{tabular}
     \caption{PD Code for a matching edge and virtual crossing. }
     \label{fig:PDCode}
\end{center}
\end{figure}

For the perfect matching graphs shown in \Cref{fig:PetJ3} we observe that each graph has (non-unique) graph PD notation given by:
\begin{eqnarray*}
Pet & := & G[M[9,10,1,5],M[6,7,2,1],M[8,9,3,2],M[3,4,6,10],M[4,5,8,7], V[5,9],V[3,9],\\
\ & \ & V[3,5]]\\
J3 & := & G[M[5,6,9,1],M[4,5,12,10],M[11,12,1,2],M[6,7,2,3],M[7,8,10,11],M[3,4,8,9],\\
\ & \ & V[3,6],V[3,9],V[6,9]].
\end{eqnarray*}
The code above is for the Petersen graph and $J_3$ with the perfect matchings shown in \Cref{fig:PetJ3}.  For comparison with \Cref{ex:Petersen}, one may wish to calculate the polynomial for the blowup of the Petersen graph, whose graph PD code is given below:
\begin{eqnarray*}
PetBU & = & G[M[4, 5, 1, 2], M[13, 14, 3, 1], M[7, 8, 6, 4], M[10, 11, 14, 15],M[29, 30, 17, 18],\\
 \ & \ &  M[27, 25, 5, 6], M[20, 21, 2, 3],M[24, 22, 12, 10], M[21, 19, 18, 16], M[26, 27, 23, 24],\\
 \ & \ & M[9, 7, 11, 12],M[16, 17, 8, 9], M[19, 20, 22, 23],M[25, 26, 30, 28], M[28, 29, 15, 13],\\
\ & \ & V[2, 11],V[2, 12], V[3, 11], V[3, 12], V[26, 29], V[26, 30], V[27, 29], V[27, 30]]
\end{eqnarray*}

After copying and pasting the code on the next page to a Mathematica notebook, along with the PD notation above, one may run the calculation on the Petersen graph by entering $T[Pet]$.  Both $Pet$ and $J3$ will run (almost) instantaneously on modern hardware. We see the biggest improvement in time using \Cref{thm:PKEqualsT} with the blowup of the Petersen graph, $PetBU$. The old method (calculating several filtered $n$-color homologies) took over a week of computation to calculate, while $T(PetBU)$ using the PD notation above takes a little over 9 hours on an M2 Pro MacBook. 

\newpage

\noindent{Mathematica Code:}\\

\noindent \verb$rule0 = {V[x_, y_] :> (2 Node[x, y] - 1)};$\\
\verb$rule1 = {M[a_, b_, c_, d_] :> arc[a, d] arc[b, c] - arc[a, c] arc[b, d]};$\\
\verb$rule2 = {Node[x__] Node[y__] :> Node @@ Union[List[x], List[y]] /;$\\
\verb$Intersection[List[x], List[y]] != {}, arc[a_, b_] arc[b_, c_] Node[x__] :> $\\
\verb$arc[a, c] (Node[x] /. b :> Min[a, c] // DeleteDuplicates) /; $\\
\verb$MemberQ[List[x], b], arc[a_, b_] arc[c_, b_] Node[x__] :> $\\
\verb$arc[a, c] (Node[x] /. b :> Min[a, c] // DeleteDuplicates) /; $\\
\verb$MemberQ[List[x], b], arc[b_, a_] arc[b_, c_] Node[x__] :> $\\
\verb$arc[a, c] (Node[x] /. b :> Min[a, c] // DeleteDuplicates) /; $\\
\verb$MemberQ[List[x], b], arc[b_, a_] arc[c_, b_] Node[x__] :> $\\
\verb$arc[a, c] (Node[x] /. b :> Min[a, c] // DeleteDuplicates) /; $\\
\verb$MemberQ[List[x], b]};$\\
\verb$rule3 = {arc[a_, b_] arc[b_, c_] :> arc[a, c], arc[a_, b_] $\\
\verb$arc[c_, b_] :> arc[a, c], arc[b_, a_] arc[b_, c_] :> arc[a, c], $\\
\verb$arc[b_, a_] arc[c_, b_] :> arc[a, c]};$\\
\verb$rule4 = {(arc[a_, b_])^2 Node[x__] :> c[a] (Node[x] /. b :> a $\\
\verb$// DeleteDuplicates) /; MemberQ[List[x], b] && a < b,$\\
\verb$(arc[a_, b_])^2 Node[x__] :> c[b] (Node[x] /. a :> b $\\
\verb$// DeleteDuplicates) /; MemberQ[List[x], a] && b < a, $\\
\verb$Node[x__] Node[y__] :> Node @@ Union[List[x], List[y]] /;  $\\
\verb$Intersection[List[x], List[y]] != {}};$\\
\verb$rule5 = { (arc[a_, b_])^2 :> c[Min[a, b]], arc[a_, a_] :> c[a]};$\\
\verb$rule6 = {Node[x__]^m_ :> Node[x]};$\\
\verb$rule7 = {Node[x__] :> n (Product[c[List[x][[i]]]^(-1), {i, 1,$\\
\verb$Length[List[x]]}])};$\\
\verb$rule8 = {c[a_] :> n};$\\
\verb$T[t_] := Simplify[((Product[t[[i]], {i, 1, Length[t]}] /. rule0 // $\\
\verb$Expand) /. rule6 /. rule1 // Expand) //. rule2 //. rule3 //. $\\
\verb$rule4 //. rule5 /. rule6 //. rule7 //. rule8];$\\

\end{document}